\documentclass[11pt]{article}
\usepackage{amssymb,amsmath,amsthm,latexsym}
\usepackage{amsfonts}
\usepackage{epsfig}
\usepackage{epstopdf}
\usepackage[tableposition=below]{caption}
\usepackage{cite}
\usepackage[english]{babel}
\usepackage{mathtools}
\usepackage{amssymb}
\usepackage{geometry}
\usepackage{blindtext}
\usepackage{titlesec}
\usepackage{graphicx}
\usepackage{amsthm}
\usepackage{graphics,epsfig,graphicx}
\usepackage{caption}
\usepackage{subcaption}
\usepackage[colorlinks=true, allcolors=blue]{hyperref}
\usepackage{enumitem}
\usepackage{titlesec}
\usepackage{pdfpages}
\usepackage{textcomp}
\usepackage{array}
\usepackage{multirow}
\usepackage{hhline}
\usepackage{longtable,rotating}
\newtheorem{theorem}{Theorem}[section]

\newtheorem{corollary} {Corollary}

\newtheorem{lemma}{Lemma}
\newtheorem{prop}{Proposition}

\begin{document}
\title{
 A Graph Theoretical Approach to Optimizing Minimum Italian Domination Sets }
\author{ Muhammad Zeeshan$^{1}$, Nahid Akhtar$^{2}, $ Muhammad Faisal Nadeem$^{3}$ }
\date{}
\maketitle
\vspace{-7mm}
\begin{center}
{\it\small
$^1$Department of Mathematics,\\ Government College University Lahore, Pakistan\\

$^2$Department of Mathematics,\\ Government College University Lahore, Pakistan\\

$^3$Department of Mathematics,\\ COMSATS, University Islamabad Lahore Campus, Pakistan.

{\it E-mail:}  m.zeeshan@gcu.edu.pk, nahidakhtar@gcu.edu.pk, mfaisalnadeem@ymail.com }
\end{center}

\begin{abstract}
A classical problem in graph theory, known as the Italian domination number (also called Roman 2-domination number), involves assigning labels of $0$, $1$, or $2$ to each node $v$. The goal is to ensure that every node with a label of $0$ has a sum of labels of the nodes in its closed neighborhood that is $2$ or greater. In computer systems, It coined encompassing a robust cybersecurity strategy that will protect networks from potential threats, such as hacking, malware, and unauthorized access, by deploying security measures to provide the highest level of protections while reducing the misuse of resources. Toeplitz graphs are a special kind of graphs built over Toeplitz matrices from linear algebra, which are matrices with constant straight diagonal members. In this paper, we provide a detailed analysis regarding the Italian domination numbers for every Toeplitz graph family. We provide comprehensive results on Italian domination numbers across multiple graph families and identify the specific values at which the Italian domination number alters with increasing generator values.
\end{abstract}

 \section{Introduction}

A dominant set of a graph $G = (V, E)$ is a subset $T \subseteq V$ such that every node $v \in V$ is either a node of $T$ or adjacent to a node of $T$. Italian domination function, i.e. IDF, is a very special domination parameter which is the generalization of Roman dominating function. In graph theory, a Roman domination function on a graph $G = (V, E)$ is a function $f : v \rightarrow \{0, 1, 2\}$ that has a specific property. This property states that for every node $w$ in the graph, if $f(w) = 0$, then $w$ must be adjacent to at least one node $v$ where $f(v) = 2$. This condition ensures a certain level of connectivity and dominance within the graph. In this context, we examine the weight of a Roman dominating function, which is represented by the value $f(z) = \sum_{w\in z} f(W)$. The Roman domination number, denoted as $\Upsilon_{R} D(G)$ \cite{3,8,9}, refers to the least weight of a Roman dominating function of a graph $G$. In this analysis, we aim to establish the prevalence of Italian domination on the graph $G = (V, E)$. To achieve this, we introduce a function $f : V \rightarrow \{0, 1, 2\}$ that assigns values to each node $v$. It is worth noting that for any node $v$ with $f(v) = 0$, the sum of the values of its neighboring nodes $u$, denoted as $N(v)$, must be greater than or equal to $2$.
The value of an IDF weight is calculated as the sum of all the values in $V$ for $f(v)$. When considering this problem from a graph labeling perspective, it is important to note that any node labeled $0$ should have the sum of the labels of its neighboring nodes be at least $2$. Exploring the least weight of IDFs is the Italian domination number, denoted by IDN, is a concept related to $G$. When considering the double Roman Empire, a crucial defense strategy involves ensuring that every location without a legion is adjacent to another location without a legion. This strategic approach helps to minimize the overall expenses of the empire \cite{21,22,26}.\newline
Over the past thirty years, there has been a consistent and impressive growth in their usage. This can be attributed to their wide range of applications, which span both theoretical and real-world problems. These include tackling challenges like facility location problems and developing strategies for city defense, among others. Several parameters related to domination have been extensively studied, including the Roman\cite{2}, double Roman\cite{15}, Roman $3$-domination\cite{16} and in this active area of study different researchers investigated the domination numbers in graphs with a focus on Roman, Italian, and double Italian domination parameters\cite{4,23,27}.  In the context of Roman domination, the deployment of Roman legions is limited to a maximum of two legions at any given location\cite{11}. Several studies have been conducted on domination numbers in relation to Italian domination. These include the perfect Italian domination, outer-independent Italian domination number, independent Italian domination number and global Italian domination number \cite{7,10,12}. Independent domination in subcubic and cubic graphs studied in \cite{1,5,8}.

Toeplitz matrices are characterized by having equal elements along each descending diagonal from left to right. One interesting property that a symmetric Toeplitz matrix possesses is its compliance with a certain condition. This property holds along with the symmetry condition $T_{ij}= T_{ji}$ and has entries in form of $0$ or $1$. The graphs constructed through Toeplitz matrices called Toeplitz graphs, $T_{G}$. A symmetric matrix with dimensions $n\times n$ (square matrix) is called Toeplitz matrix if  $T_{ij}= T_{ji}$ for each  $i, j$ goes $1$ to $n-1$ and can be shown as follows:

\[
  T =
  \left[ {\begin{array}{ccccc}
    A_{0} & A_{1} & A_{2} & \cdots & A_{n-1}\\
    A_{1} & A_{0} & A_{1} & \cdots & A_{n-2}\\
    A_{2} & A_{1} & A_{0} & \cdots & A_{n-3}\\
    \vdots & \vdots & \vdots & \ddots & \vdots\\
    A_{n-1} & A_{n-2} & A_{n-3} & \cdots & A_{0}\\
  \end{array} } \right]
\]

 Here, $A_{i}$, $i = 0,1,2,\cdots n-1$ are the entries of matrix with $A_{0} = 0$ and all other remaining entries are either $0$ or $1$ \cite{25}.\newline
Protecting sensitive data from cyber attacks that might disrupt, damage, or destroy critical information is highly challenging in a complex network system. This is due to the fact that cyber threats possess the capability to destroy any type of data. Toeplitz matrices are highly valuable because they may be used to accurately estimate the direction of arrival of any attack. This attribute renders them extremely adaptable. As a result, networks constructed using Toeplitz matrices provide effective approaches for processing and optimising memory usage. This is because the construction of the remaining network only relies on the initial row and column.

Different researchers around the world play vital role in the study of Toeplitz graph and finds its hidden properties and solve many problems. Where, Sara nicoloso in \cite{19} studied connectivity of graph and find chromatic number for $3$ values, Ahmad mojallal in \cite{18} characterized chordal Toeplitz graphs and find clique cover number for edges and vertices, Faisal nadeem in \cite{20,24} provide general results on hamiltonian connectedness and metric dimension for generalized generating values, and the hamiltonicity in directed Toeplitz graphs with the help of restricted generators presented in \cite{17} by Shabnam malik.

The aim is to analyze the characteristics and properties of these domination numbers and their impact on graph structures. By examining the dominance relationships within these graphs, we aim to gain a deeper understanding of their behavior and potential applications. Through the execution of this study, we anticipate making a valuable addition to the existing knowledge in this field and producing valuable insights that can be utilised for future research. Our research presents a more resilient manifestation of Roman dominance. This iteration of Roman rule enhances security by guaranteeing that any assault can be effectively repelled by a minimum of two legions.
 In the subsequent discussion, we will explore how the deployment of three legions at a specific location enhances defense capabilities, offering a stronger and more flexible approach while keeping costs lower than expected. As we delve into the topic at hand, it is important to establish a solid foundation of knowledge. In this exploration, we will embark on a journey of discovery and analysis. Through meticulous research and thoughtful examination, we aim to shed light on the subject matter. In this study, we present the findings that highlight a fundamental characteristic of the Italian domination number.
\begin{theorem}
   For every graph $G$, $\Upsilon_{I}$$(G)$ $\leq$ $2\Upsilon$$(G)$.\cite{6}
\end{theorem}
\begin{theorem}
  If $G$ is a connected graph with maximum degree $\Delta(G)$, then
  $\Upsilon_{I} (G) \leq$ ${\dfrac{2|V(G)|}{\Delta(G)+2}}$.\cite{4}

\end{theorem}
In this paper, we will delve into the ongoing investigation of Italian dominating functions in graphs. In this study, we investigate connected Toeplitz graphs and determine the maximum order they can achieve while minimizing the Italian domination number. In this study, we also aim to analyze connected Toeplitz graphs with minimum generating values of at least two to optimize the values of generators $t_{1}, t_{2}, t_{3}, \cdots , t_{k}$, while minimizing the Italian domination number, a comprehensive approach is required.

\section{Italian Domination Function on $T_{n}<t_{1}, t_{2}, t_{3}, \cdots, t_{k}>$}

\begin{lemma}

\[\Upsilon_I (T_{n}<1,2>) =\begin{cases*}
2 & when $n = 3,4,5$,\\
3 & when $n = 6,7$,\\
\Big \lfloor \dfrac{n-2}{3} \Big \rfloor + 2& when $n \geq 8$.

\end{cases*} \]

 \end{lemma}

\begin{proof}

There are three possible cases for this lemma.\newline
\textbf{{Case $1$:}} Consider a graph, denoted as $T_{G}$, with orders $3$, $4$, and $5$. Consider the node $u$, which has the highest degree and is labeled with $2$. Additionally, all the adjacent nodes are labeled with $0$. As a result, we obtain an Italian domination number of $2$.\newline
IDS = $\{u_{3}\}$\newline
\textbf{{Case $2$:}} Let $T_{G}$ be a graph of order $6$ and $7$. Assign label $1$ to $u_{1}$, $u_{n-3}$, $u_{n}$ and assign $0$ to all remaining nodes of $G$. Thus, f$(u_{1})$ + f$(u_{n-3})$ + f$(u_{n})$ = $3$. So IDN = $3$ and IDS = $\{u_{1}, u_{n-3}, u_{n}\}$.\newline
\textbf{{Case $3$:}} We will use induction to prove this case.\newline
Base Case: put $n = 8$. Suppose that f$(u_{1})$, f$(u_{4})$, f$(u_{6})$ and f$(u_{7})$ has value $1$ under $f$. However, assign $0$ to all other neighbors of already assigned nodes. The sum of values assigned to all nodes is four. Therefore, IDN = $4$.\newline
Inductive Case: If the result is true for $n-1$, then it must be true for $n$. Assume that the nodes $u_{3k+1}$ for all $k = 0,1,2,3, \cdots , n$ be the set of nodes labelled by $1$. For $u_{i} \notin u_{3k+1}$, we assign $0$. In addition, for the graph of order $3k+2$ for $k > 2$, the nodes $u_{3k+1}$ and $u_{n}$ lebelled $1$ and all other $u_{j} {\notin u_{3k+1}}$ are labelled $0$. Therefore, $\Upsilon_I (T_{n}<1,2,3>)$ = $\Big \lfloor \dfrac{n-2}{3} \Big \rfloor + 2$.

\end{proof}

\begin{lemma}

\[\Upsilon_I (T_{n}<1,2,3>) =\begin{cases*}
2 & when $n = 4,5,6,7$,\\
\Big \lfloor \dfrac{n-2}{4} \Big \rfloor + 2& when $n \geq 8$.

\end{cases*} \]

 \end{lemma}

 \begin{proof}
   There are two possible cases for this lemma.\newline
\textbf{{Case $1$:}} Let $G$ be a graph of order $4, 5, 6$ and $7$. Let $u$ be the node of maximum degree labeled by $2$ and all other adjacent nodes labelled $0$. Consequently, we get Italian domination number $2$.\newline
IDS = $\{u_{4}\}$\newline
\textbf{{Case $2$:}}Inductive Case: If the result is true for $n-1$, then it must be true for $n$. Assume that there are nodes $u_{1}, u_{2}, u_{3}, \cdots, u_{n-1}, u_{n}$ and for all $u_{4k+1}$ where $k = 0,1,2,3, \cdots , n$ be the set of nodes labelled by $1$. The other nodes that are not equal to $u_{4k+1}$ for all $k = 0,1,2,3, \cdots , n$ labelled $0$. Thus, by using this sequence Italian domination function can easily obtained. Therefore, $\Upsilon_I (T_{n}<1,2,3>)$ = $\Big \lfloor \dfrac{n-2}{4} \Big \rfloor + 2$ and IDS = $\{u_{1}, u_{5}, u_{9}, \cdots, u_{n}\}$.

 \end{proof}

\begin{theorem}

\[\Upsilon_I (T_{n}<t_{1},t_{2}, t_{3}, \cdots, t_{k}>) =\begin{cases*}
2 & when $t_{i}+1\leq n \leq 2t_{i}+1$,\\
3 & when $n = 2t_{i}+2, 2t_{i}+3$,\\
\Big \lfloor \dfrac{n-2}{t_{i}+1} \Big \rfloor + 2& when $ n \geq 2t_{i}+4$.

\end{cases*} \]

 \end{theorem}

 \begin{proof}
   This result can be obtain by using three possible steps.\newline
   \textbf{{Case $1$:}} Let $G$ be a graph of order $t_{i}+1$ to $2t_{i}+1$. Suppose minimum and maximum connections of any node in this case is $t_{i}$ and ${n-1}$ respectively. Here, the node of maximum degree $u_{t_{i+1}}$ has connection with all other nodes. So, labelled $2$. Consequently, Italian domination number and set are $2$ and $\{u_{t_{i+1}}\}$ respectively.\newline
   \textbf{{Case $2$:}} Let $G$ be a graph of order  $2t_{i}+2$ and $2t_{i}+3$. Assign label $1$ to the nodes $u_{1}$,$u_{t_{i}+2}$ and $u_{n}$ and assign $0$ to all the remaining nodes of graph. Thus, we obtain the function f$(u_{1})$ + f$(u_{t_{i+2}})$ + f$(u_{n})$ = $3$. So, IDN = $3$ and IDS = $\{u_{1}, u_{t_{i+2}}, u_{n}\} $.\newline
   \textbf{{Case $3$:}} We proceed this case by using induction.\newline
Base Case: put $n = 2t_{i}+4$ for any $i$ from $3$ to $n-1$. Suppose the nodes $u_{1}$, $u_{t_{i+2}}$, $u_{n-1}$ and $u_{n}$ have value $1$ under $f$. So, f$(u_{1})$ + f$(u_{t_{i+2}})$ + f$(u_{n-1})$ + f$(u_{n})$ = $4$. Since, by using formula, we conclude that IDN = $4$ and IDS = $\{u_{1}, u_{t_{i+2}}, u_{n-1}, u_{n}\}$.\newline
Inductive Case: In the instance if the result is valid for $n-1$ nodes, then it is assumed that it is valid for all $n$ nodes. To proceed with all possible values of $n-1$, take any $t_{i} \geq 3$ at the position of the last generator $t_{k}$ and proceed it for all possible values of $n$. Take into consideration that $t_{m}$ represents the total number of generators in a family. It is important to emphasise that there is a distinct category of nodes that possess the least and maximum degree in the same sequence across the entire series of $t_{1},t_{2}, t_{3}, \cdots, t_{k}>$ for all possible values of $n-1$. Since, the nodes  $u_{1}$ and $u_{n}$ are of minimum degree and the nodes from $u_{t_{i+1}}$ to $u_{{n-t_{i}}}$ have maximum degree. Assign $1$ to the sequence of nodes $u_{1}$, $u_{t_{m+2}}, u_{2t_{m+3}}, u_{3t_{m+4}}, \cdots, u_{n}$ and $0$ to all other vertices in the graph. Moreover, by using this sequence Italian domination number of all possible values of $n-1$ could be obtained for any $t_{m}$ and  $t_{i}$, where $i \geq 3$. Consequently, by using previous cases and lemma $1$ and $2$ result is hold for all values of $n$ and for any $t_{i}$. Thus, IDN = $\Big \lfloor \dfrac{n-2}{t_{i}+1} \Big \rfloor + 2$ for all $ n \geq 2t_{i}+4$ and IDS =  $\{u_{1}$,  $u_{t_{m+2}}, u_{2t_{m+3}}, u_{3t_{m+4}}, \cdots, u_{n}\}$.\newline
\textbf{Conversely},

Since, order of graph varies from $t_{i}+1$ to $2t_{i}+1$. Suppose on contrary,  $\{u_{t_{i+1}}\}$ is not an Italian domination set. Then, there is no other node that is connected to all other nodes in the graph except $u_{t_{i}}$. But for $n= 2t_{i}+1$, the node $u_{t_{i}}$ is not connected to $u_{n}$. So, $u_{t_{i}}$ is not a node of maximum degree.\newline
Secondly, let $u_{1}\notin$ IDS and the node $u_{2} \in$ IDS. Then there is a node $u_{t_{i}}+2$ of maximum degree. Here, sum of all neighbors of $u_{1}$ is $1$. So, we need at least one more node $u_{4}$ to obtain desired condition. But by including $u_{4}$ the set is not minimum. Therefore, it is not possible to obtain minimal Italian domination set by using these vertices. Thus minimum IDS =  $\{u_{1}, u_{t_{i+2}}, u_{n}\} $.\newline
For the final step, take any $t_{i} \geq 3$ at the position of the final generator $t_{k}$ and continuing with the process until you have exhausted all of the possible values of $n$. Based on the fact that the  deg$(u_{1})$ = deg$(u_{n})$ = $t_{m}$. Given that the node $u_{1}$ does not have the label $1$, and that the function on the node f$(u_{2})$ is equal to $1$. The Italian domination property is not satisfied when the total of all the neighbours of $u_{1}$ proves to be equal to zero.  For another example, let us suppose a scenario in which the node $u_{1}$ is also not labelled $1$, and the function on the node f$(u_{3})$ is equal to $1$. In addition, the sum of all the neighbours of $u_{1}$ does not meet the Italian domination property like the previous example. It means any other sequence like,  $\{u_{2}$,  $u_{t_{m+2}}, u_{2t_{m+3}}, u_{3t_{m+4}}, \cdots, u_{n}\}$ or  $\{u_{3}$,  $u_{t_{m+2}}, u_{2t_{m+3}}, u_{3t_{m+4}}, \cdots, u_{n}\}$ need at least one more node to satisfy the property of minimum Italian domination set.\newline
Thus, we need special sequence of the form $\{u_{1}$,  $u_{t_{m+2}}, u_{2t_{m+3}}, u_{3t_{m+4}}, \cdots, u_{n}\}$ to obtained minimum Italian dominating set.\newline
For the detailed explanation, figure are presented.

\begin{figure}[h]
  \centering
  \includegraphics[angle = 270, width=.25\linewidth]{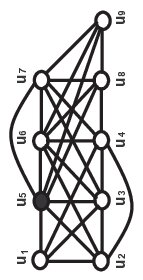}
  \caption{$T_{9}<1,2,3,4>$}\label{1}
\end{figure}

 In figure $1$, graph of order $9$ under four generators are presented and the assigned value of node $u_{5}$ is $2$.

\begin{figure}[h]
  \centering
  \includegraphics[angle=270, width=.25\linewidth]{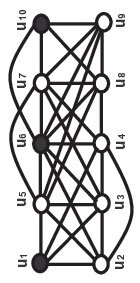}
  \caption{$T_{10}<1,2,3,4>$}\label{1}
\end{figure}

   \begin{figure}[h]
\centering
\begin{subfigure}{.50\linewidth}
\centering
\includegraphics[angle=270, width=.60\linewidth]{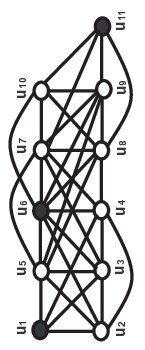}
\caption{Toeplitz graph $T_{11}<1,2,3,4>$}\label{fig:sub13}
\end{subfigure}%
\begin{subfigure}{.45\textwidth}
\centering
\includegraphics[angle=270, width=.60\linewidth]{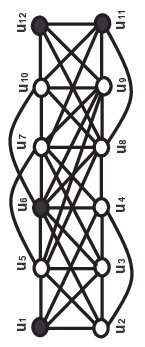}\vspace{0.17cm}
\caption{Toeplitz graph $T_{12}<1,2,3,4>$}\label{fig:sub14}
\end{subfigure}
\caption{Toeplitz graph $T_{n}<1,2,3,4>$ of order $11$ and $12$  }\label{fig:12}
\end{figure}

   In the figures $2$ and $3$ the graphs of order $10$, $11$ and $12$ are presented and all black nodes have values of $1$.

 \end{proof}

\begin{theorem}
 Let for any $n > t$ and $t$ is odd, then $\Upsilon_{I} (T_{n}<1,t>$) = $ \Big \lfloor \dfrac{n}{2} \Big \rfloor $.

\end{theorem}

\begin{proof}
  We obtain this result by using Induction.\newline
  \textbf{{Case $1$:}} Let's choose initial value of  $t = 3$. If $2n \equiv 0$ (mod $3$) for $n > 1$ then f$(u_{1})$, f$(u_{3})$, $\cdots$ and f$(u_{n-1})$ has value $1$ under $f$ and assign $0$ to all other nodes and if   $3n-1 \equiv 1$ or $2$ (mod $3$) for $n > 1$ then f$(u_{2})$, f$(u_{4})$, $\cdots$, f$(u_{n-1})$ has value $1$ under $f$ and assign $0$ to all other nodes. Therefore, by using $2n$ and $2n+1$ for $n > 1$ then Italian domination number can compute and increased by $1$ for every value of $n$.\newline
 \textbf{ {Case $2$:}} Assume that $u_{1}, u_{2}, u_{3} \cdots, u_{j-1}, u_{j}$ be a sequence of nodes of $n$ for $j < n$. Let $n > t$ and choose $t = 5$ then $n \geq 6$. Therefore, by using $2n$ and $2n+1$ nodes for $n > 2$, then f$(u_{1})$, f$(u_{3})$, $\cdots$ and f$(u_{n-1})$ has value $1$ under $f$ and assign $0$ to all other nodes, and f$(u_{2})$, f$(u_{4})$, $\cdots$, f$(u_{n-1})$ has value $1$ under $f$ and assign $0$ to all other nodes respectively. Thus, Italian domination number can compute and increased by $1$ for every value of $n$.\newline
  Consider the result is true for all values of $t$. It is clearly seen from above arguments if result is true for all values of $t$ it means it hold for all $n$. Therefore,  we generalize the result by using the fact of nodes $2n$ and $2n+1$ for $n > \dfrac{t-1}{2} $, then f$(u_{1})$, f$(u_{3})$, $\cdots$, f$(u_{n-1})$ has value $1$ under $f$ and assign $0$ to all other nodes, and f$(u_{2})$, f$(u_{4})$, $\cdots$, f$(u_{n-1})$ has value $1$ under $f$ and assign $0$ to all other nodes respectively. By using the above arguments, we clearly obtain the result with increment of $1$ for each next value $n$. Consequently, result is true for all values of $n$ and $t$.

\end{proof}

\begin{theorem}
 Let for any $n > t$ and even $t > 2$, Then

\[\Upsilon_I (T_{n}<1,t>) =\begin{cases*}
\Big \lceil \dfrac{t + 2}{2} \Big \rceil & when $n = t + 1, t + 2$,\\
\dfrac{t+4}{2} & when $t+3 \leq n \leq t+5$,\\
\Big \lfloor \dfrac{n}{2} \Big \rfloor & when $n \geq t + 6$.

\end{cases*} \]

\end{theorem}

\begin{proof}
  This theorem proceed in three different steps.\newline
\textbf{{Case $1$:}}  Define a function $f$ on $n = t+1$, then $f(u_{i}) = 1$ for $i = 1, 3,\cdots, n$ and $0$ otherwise. The function $f$ on $n = t+2$ have $f(u_{i}) = 1$ for $i = 2, 3, \cdots, n-1$ and $0$ otherwise. Therefore result is obtain and $f$ be an Italian domination function.\newline
 To show converse, let the function $f(u_{i}) = 1$ for odd $i$ and  $0$ when $i$ is even. The function defined on nodes that have labelled $1$ must have $u_{n} = 0$. For this way of assigning the criteria of Italian domination function is not fulfilled. Consequently, we need $u_{n} = 1$.\newline
\textbf{{Case $2$:}} Let $V = u_{1}, u_{2}, u_{3} \cdots, u_{n-1}, u_{n}$ be the nodes of graph. Associate the class of double nodes for two different values of $i$ to the vertices that have positive value. If $ u_{i}, \cdots, u_{n-3}, u_{n-1} = 1$
be the positively labelled nodes for $i = 1$ or $2$ under two initial choices.\newline
 Conversely, we claim that, if function value varies from $t+2 \leq n \leq t+5$ then for initial value Italian domination function can obtain by using Italian domination set of three nodes. Here, $\Upsilon_I \neq \dfrac{t+4}{2}$. On the other hand, suppose $t+3 \leq n \leq t+6$. If $f(u_{n}) = 1$ or $2$, then $f(u_{n-1}) = 0$ and $f(u_{n-2}) \geq 1$. Define $g: t+6 \rightarrow \{2, 5, 7, 8, 10\}$. Thus, $g(u_{i}) \neq \dfrac{t+4}{2}$.\newline
\textbf{{Case $3$:}} Assume that $u_{1}, u_{2}, u_{3} \cdots, u_{j-1}, u_{j}$ be a sequence of nodes of vertices of $n$ for $j < n$. We proceed this case by using Induction. Let $n \geq t+6$ and choose $t = 4$ then $n \geq 10$. Here, two different classes for $n$ is exist, one of which is $n \equiv 0, 1$(mod $t$) and the other one is $n \equiv 2, 3$ (mod $t$) have same Italian domination number with increment of $1$ for the next step. Suppose result is true for the next value $t$. Here, if $f(u_{n}) = 0$, then $f(u_{n-1})$ and $f(u_{n-t})$ must be $1$. Therefore, three possible classes: $n \equiv 0, 1 $(mod $t$), $n \equiv 2, 3$ (mod $t$) and $n \equiv 4, 5$ (mod $t$) have same numbers in Italian dominating set for the each next with increment of $1$.\newline
Consider the result is true for all values of $t$. It is clearly seen from above arguments if result is true for all values of $t$ it means it hold for all $n$. Here, we generalize the argument in the form of $n \equiv 0, 1$ (mod $t$), $2, 3 $(mod $t$) $\cdots, t-2, t-1$(mod $t$), where $t < n$. By using the above arguments, we clearly obtain the result with increment of $1$ for each next value $n$. Consequently, result is true for all values of $n \geq t+6$.\newline
Foe the converse part, assuming that $n$ is not congruent to $0$ or $1$ modulo $t$, we consider the Italian domination function $f(v_{0}), f(v_{1}), f(v_{2})$ for $t = 4$. It is worth noting that a vertex with a weight of $0$ will always have a neighboring vertex with a weight of either $1$ or $2$. If $n \not\equiv 0, 1 (mod$ $t)$ then there are two adjacent nodes $u_{i}$ and $u_{i+1}$ such that their weights are zero. Let $f(u_{i}) \geq 1$ and $f(u_{i+1}) = 0$, then $f(u_{i+2}) > 0$. Thus, if $n$ is not congruent to desired conditions then there exist at least two nodes having labelled $0$ they do not fulfill the criteria od Italian domination function.\newline
\end{proof}

\begin{corollary}
Let $2t \leq n \leq 2t+1$, then $\Upsilon_I $ $(T_{n}<1,t>)$ = $t$ for $ n \equiv 0,1 \pmod t$
\end{corollary}

\begin{theorem}

Let for all $ n \geq t+3$ and for any $ t \equiv 1$ (mod $3$), then $\Upsilon_I (T_{n}<1,2,t>) = \Big \lceil \dfrac{n}{3} \Big \rceil $.

\end{theorem}

\begin{proof}

  We proceed by induction, By using minimum value of $t = 4$, which implies result hold for all values of $n \geq 7$. Let $u_{1}, u_{2}, u_{3} \cdots, u_{n-1}, u_{n}$ be the nodes for initial value of $t$. When $n \equiv 1, 2$ (mod $t$) then $u_{2}, u_{5} \cdots, u_{j}, \cdots, u_{n-2}, u_{n}$ = $1$ and if $n \equiv 0, 3$ (mod $t$) then $f(u_{1})$, $f(u_{4})$, $\cdots$, $f(u_{j}), \cdots$, $f(u_{n})$ are labelled $1$, where $j$ is appropriate node between these values. By assigning $0$ to all other vertices, result is obtained.\newline
  Let $t = 7$ which is congruent to $1$ mod $3$ gives $n \geq 10$. When $n \equiv 0, 3, 6$ (mod $t$) then $f(u_{1})$, $f(u_{4})$, $\cdots$, $f(u_{n-3})$, $f(u_{n})$ are labelled $1$ and if  $n \equiv 1, 2, 4, 5$ (mod $t$) then $f(u_{2})$, $f(u_{5})$, $f(u_{8})$, $\cdots$, $f(u_{n-4})$, $f(u_{n-1})$ are labelled $1$ and all other vertices labelled $0$. Let us suppose, the result is true for  all finitely many values $t$. The induction hypothesis obtain by proceeding in the similar way. If $n \equiv 0, 3, \cdots, (n-1)$ (mod $t$) then $f(u_{1})$, $f(u_{4})$, $\cdots$, $f(u_{j})$, $\cdots$, $f(u_{n})$ are labelled $1$ and other nodes labelled $0$ and when $n \not\equiv 0, 3, \cdots, (n-1)$ (mod $t$) then $f(u_{2})$, $f(u_{5})$, $f(u_{8})$, $\cdots$, $f(u_{j})$, $\cdots$, $f(u_{n-4})$, $f(u_{n})$ are labelled $1$, where $j$ is appropriate vertex and remaining nodes labelled $0$. Consequently, the result is prevailed for all values.\newline
 \textbf{ Conversely}, suppose that $t \neq 3s+1$ for $s \geq 1$. Then there are two possibilities: either  $t > 3s+1$ or $t < 3s+1$. \newline
  {Case $1$:} Let  $t > 3s+1$ for $s \geq 1$. Then deg$(u_{1})$ = deg$(u_{n})$ = $3$ and all other nodes have degree $\geq 3$. Clearly, for the minimum or maximum value of $t$ the Italian domination set is $\{u_{1}, u_{4} \cdots, u_{j}, \cdots, u_{n}\}$ where $j$ is any appropriate middle value belong to the set. So, $\Upsilon_{I}$ must be greater than or equal to $\Big \lceil \dfrac{n}{3} \Big \rceil $.\newline
  {Case $2$:}  Let  $t < 3s+1$ for $s \geq 1$. Then deg$(u_{1})$ = deg$(u_{n})$ = $3$ and assign $1$ to both of these vertices. Let's take $n = 13$ then Italian domination number is $5$ but $\{u_{1}, u_{5}, u_{9}, u_{13}\}$ be the Italian domination set. Clearly, $\Upsilon_{I}$ must be less than or equal to $\Big \lceil \dfrac{n}{3} \Big \rceil $.\newline
  Consequently, $\Upsilon_I (T_{n}<1,2,t>) = \Big \lceil \dfrac{n}{3} \Big \rceil $.

\end{proof}

\begin{theorem}
Let for all $ n \geq t+3$ and for any $ t \equiv 2$ (mod $3$), then

\[\Upsilon_I (T_{n}<1,2,t>) =\begin{cases*}
\Big \lceil \dfrac{n}{3}  \Big \rceil + 1 & when $n = t+3+3k, k\geq 0$,\\
\Big \lfloor \dfrac{n}{3} \Big \rfloor + 1 & when $n \geq t + 3$.

\end{cases*} \]

\end{theorem}

\begin{proof}
  We prove this theorem in $2$ steps.\newline
    \textbf{Step} \textbf{$1$}: Let any $t \equiv 2$ (mod $3$) for all $n = t+3+3k, k\geq 0$ and ${u_{1}, u_{2}, \cdots, u_{n-1}, u_{n}}$ be the nodes of graph. By varying value of $k$ the values of $n$ are $t+3, t+6, t+9 \cdots, n-2$. By applying Italian domination function and labelled $1$ to $f(u_{1})$, $f(u_{4})$, $f(u_{7})$, $\cdots$, $f(u_{n-2})$ and all other remaining nodes labelled $0$. Thus, by using this chain Italian domination number can be obtain and IDS = $\{u_{1}, u_{4}, u_{7}, \cdots, u_{n-2}\}$. \newline
   \textbf{Step} \textbf{$2$}: Let  $ n \geq t+3$ and here all values of $n$ that are in step $1$ are excluded. Any $ t \equiv 2$ (mod $3$), like $5, 8, 11, \cdots n-2$ have ${u_{1}, u_{2}, \cdots, u_{n-1}, u_{n}}$ number of nodes. By applying Italian domination function on other values of $n$ that are not equal to $ t+3+3k, k\geq 0$. The $f$ applied on nodes and $f(u_{1})$ = $f(u_{4})$ = $f(u_{7})$, $\cdots$, $f(u_{n})$ = $1$ and all remaining nodes labelled $0$. Consequently, IDN = $\Big \lfloor \dfrac{n}{3} \Big \rfloor + 1$ and IDS = $\{u_{1}, u_{4}, u_{7}, \cdots, u_{n}\}$.

\end{proof}

\begin{theorem}
Let for all $ n \geq t+3$ and for any $ t \equiv 0$ (mod $3$), then

\[\Upsilon_I (T_{n}<1,2,t>) =\begin{cases*}
 \Big \lceil \dfrac{n}{3} \Big \rceil & when $n < t+3$,\\
\Big \lfloor \dfrac{n}{3} \Big \rfloor + 1 & when $n \geq t + 3$.

\end{cases*} \]

\end{theorem}

\begin{proof}
  There are two possible cases.\newline
  \textbf{Case $1$}: Let  $n < t+3$. The minimum value of $t$ = $6 \equiv 0 (mod$ $3)$ and ${u_{1}, u_{2}, \cdots, u_{8}}$ be the nodes of graph. The chain of nodes  $f(u_{2})$, $f(u_{5})$, $f(u_{7})$ = $1$ and all other nodes labelled by $0$. Thus, Italian domination number is $3$. Let proceed it generally, then $f(u_{2})$, $f(u_{5})$, $\cdots$, $f(u_{i})$, $\cdots$, $f(u_{n-1})$ = $1$ and remaining nodes labelled $0$, where $i$ is any appropriate node.\newline
   \textbf{Case $2$}: Let  $n \geq t+3$ and ${u_{1}, u_{2}, \cdots, u_{n-1}, u_{n}}$ be the nodes of graph. Then the nodes,  $f(u_{1})$, $f(u_{5})$, $f(u_{8})$, $f(u_{10})$ $\cdots$, $f(u_{i})$, $f(u_{j})$, $\cdots$, $f(u_{n})$ are labelled $1$. When $ n \equiv 0$ (mod $3$), then $f(u_{1})$, $f(u_{5})$, $\cdots$, $f(u_{i})$, $f(u_{j})$, $\cdots$, $f(u_{n})$ are labelled $1$ for $i$ and $j$ are any appropriate nodes, where the remaining nodes labelled $0$. \newline
\textbf{{Conversely,}} Suppose that $t \neq 3s+3$ for $s\geq1$. Let, for the initial case $u_{i}$, $i \in Z^{+}$ be the nodes of graph till $t+2$. Since, the Italian domination number is  $\Big \lceil \dfrac{n}{3} \Big \rceil$ and if we assign $f(u_{2})$ = $f(u_{5})$ = $1$ and otherwise $0$ for the first possible value of $t$, then  $\Big \lceil \dfrac{7}{3} \Big \rceil$ or  $\Big \lceil \dfrac{8}{3} \Big \rceil$ = $3$. But Italian dominating set contain only two nodes labelled by $1$. Thus, we need at least one more node of degree $4$ in the set to fulfill the criteria. In addition to the previous case, let $u_{i}$, $i \in Z^{+}$ be the vertices of graph for $n \geq t+3$. Let $u_{1}$ and $u_{n}$ be the nodes of minimum degree where deg$(u_{2})$ = deg$(u_{n-1})$ = $4$ and deg$(u_{3})$ = deg$(u_{n-2})$ = $5$. The degree of all nodes are greater equal to $4$. Here, assign $1$ to a node of degree $3$ and all other labelled nodes are of degree greater equal to $4$. The labelled $1$ nodes are  $f(u_{1})$, $f(u_{5})$, $\cdots$, $f(u_{i})$, $\cdots$, $f(u_{n-1})$. Clearly, there exist a node $f(u_{j})$ of maximum degree for which $f(u_{j})$ = $0$ and all neighbors of $f(u_{j})$ are also zero. The Italian domination function cannot obtained. Thus, we need $f(u_{j})$ = $1$ for getting desired function.

\end{proof}

\begin{lemma}
  Let for any $t_{k}$, $k = 1,2,3, \cdots, m$ and for all $ n $, then $\Upsilon_I (T_{n}<t_{1},t_{2},t_{3}, \cdots, t_{m}, t_{m}+2>) = \Big \lfloor \dfrac{n+t_{m}-2}{\mid t_{k}\mid} \Big \rfloor + 1 $.
\end{lemma}

\begin{proof}
  Firstly, we consider a basic case $T_{n}<1,2,3,5>$. Let $u_{1}, u_{2}, u_{3}, \cdots, u_{n}$ be the nodes of graph and $t_{k}$ is the total cardinality of generators. Define a function $f$ on the sequence of following nodes: $u_{1}, u_{2}, u_{3}, \cdots, u_{n}$. Assign $1$ to these nodes and consequently f$(u_{2})$ + f$(u_{6})$ + f$(u_{10})$ + $\cdots$ + f$(u_{n})$ = $ \Big \lfloor \dfrac{n+t_{m}-2}{\mid t_{k}\mid} \Big \rfloor + 1 $. Here, for $n = t+2$, Italian dominating set is $\{4\}$ with labelled $2$.\newline
  Moreover, let $u_{1}, u_{2}, u_{3}, \cdots, u_{n}$ be the nodes of graph $ T_{n}<t_{1},t_{2},t_{3}, \cdots, t_{m}, t_{m}+2> $ for any $t_{k} \geq 5$. Since, deg$(u_{1})$ = deg$(u_{n})$ = $\mid t_{k}\mid$ and deg$(u_{2})$ must be equal to deg$(u_{n-1})$ = $\mid t_{k}\mid + 1$. So, there exist a node of maximum degree at the position of $  \Big \lceil \dfrac{u_{n}}{2} \Big \rceil $. Therefore, under a function $f$ labelled $1$ to the following nodes: f$(u_{2})$, f$(u_{t_{m}+3})$, f$(u_{2t_{m}+4})$, f$(u_{3t_{m}+5})$, $\cdots$, f$(u_{n})$ = $1$. Consequently IDN = $ \Big \lfloor \dfrac{n+t_{m}-2}{\mid t_{k}\mid} \Big \rfloor + 1 $ and IDS = $\{u_{2}, u_{t_{m}+3}, u_{2t_{m}+4}, u_{3t_{m}+5}, \cdots, u_{n}\}$. Furthermore, in this case there are some nodes at the position $u_{t_{m}+4}$ to $u_{t_{k}}$ for which IDN = $\Big \lfloor \dfrac{n+t_{m}-2}{\mid t_{k}\mid} \Big \rfloor$ and can be obtained by using similar function.

\end{proof}

\begin{prop}
   Let for any $t > 5$ and for all $ n $, when $t \equiv 1 (mod $ $4)$ then $\Upsilon_I (T_{n}<1,2,3,t>) = \Big \lceil \dfrac{n+1}{4} \Big \rceil $
\end{prop}

\begin{proof}
  Consider $t = 9$, which is congruent to $1$ mod $4$. Let $u_{1}, u_{2}, u_{3}, \cdots, u_{n}$ be the nodes of graph. When $n = t+1$ then $u_{1}$ is adjacent to $u_{2}, u_{3}, u_{4}$ and $u_{10}$ and the node $u_{2}$ is adjacent to $u_{3}, u_{4}, u_{5}$ and $u_{3}$ is adjacent to $u_{4}, u_{5}, u_{6}$ and so on. Thus the node $u_{9}$ is adjacent to $u_{10}$. Apply Italian dominating function and labelled $1$ to the following nodes:  f$(u_{2})$, f$(u_{6})$, f$(u_{10})$ = $1$. Therefore Italian domination number is $3$ and set is  $\{u_{2}, u_{6}, u_{10}\}$.\newline
   By the induction hypothesis, there exists a Italian dominating function $f$ on $u_{2}, u_{6}, u_{10}, \cdots u_{n-5}, u_{n-1}$ of $G$ such that $f$ assign a positive value $1$ to every node congruent to $2$ (mod $4)$.  Moreover, for any value of $t<n$ which is congruent to $1$ mod $4$, all nodes are connected in the similar way for all possible values of $n$. Assign $0$ to the nodes that congruent to $0,1,3$ (mod $4)$ and always allocate $1$, $u_{2}$ and to the nodes that are congruent to $2$ (mod $4)$. Then Italian domination function can be obtained. Clearly, $f$ is a IDF of $G$ such that IDN can be obtained. Therefore, IDS = $\{u_{2}, u_{6}, u_{10}, \cdots u_{n-5}, u_{n-1}\}$.
\end{proof}

\begin{theorem}
   Let for any $t_{j}$, $j = 1,2,3, \cdots, k$ and for all $ n $, when $t_{k} \equiv 1 (mod $ $c_{i})$, $i\geq 5$ then $\Upsilon_I (T_{n}<t_{1},t_{2},t_{3}, \cdots, t_{m}, \cdots, t_{k}>) = \Big \lceil \dfrac{n+c_{i}-2}{\mid c_{i}\mid} \Big \rceil $
\end{theorem}

\begin{proof}
  Assume $t_{j}$ be the generators from $t_{1},t_{2},t_{3}, \cdots, t_{m}, \cdots, t_{k}$ for $j = 1,2,3, \cdots, k$ where $k = t_{m}+2, t_{m}+3, \cdots, n-1$ and $c_{i}$ be the cardinality of total number of generators for $i\geq5$ in the graph. Since, for each generating set greater equal to $5$ there are  $u_{1}, u_{2}, u_{3}, \cdots, u_{n}$ be the vertices. When last number of the generators is congruent to $1$ (mod $c_{i})$ for any $t_{j}$, $c_{i}$ and for all $t_{k}$, $n$ then there exist special pattern in the graph. That is the nodes: deg$(u_{1})$ = deg$(u_{n})$, deg$(u_{2})$ = deg$(u_{n-1})$, deg$(u_{3})$ = deg$(u_{n-2})$, deg$(u_{4})$ = deg$(u_{n-3})$, $\cdots$, deg $\dfrac{u_{n}+1}{2}$  = $V_{\Delta_{G}}$. So, there is equal number of nodes for each degree and can be obtain by measuring forward and backward distance from initial and final vertices respectively. Apply IDF $f$ and assign f$(u_{2})$, f$(u_{t_{m}+3})$, f$(u_{2t_{m}+4})$, f$(u_{3t_{m}+5})$, f$(u_{4t_{m}+6})$, $\cdots$, f$(u_{{n}})$ = $1$ to the nodes that are $1 (mod $ $c_{i})$ and the all other nodes that are not congruent to $1 (mod $ $c_{i})$ labelled by $0$ for all $i\geq 5$. Consequently, Italian dominating function is obtained and IDS =  $\{u_{2}, u_{t_{m}+3}, u_{2t_{m}+4}, u_{3t_{m}+5}, \cdots, u_{n}\}$.\newline

   The converse can be obtained from the fact that it is impossible to substitute any node in this sequence with another node that would increase the Italian dominance number.\newline

   \begin{figure}[h]
  \centering
  \includegraphics[ width=.25\linewidth]{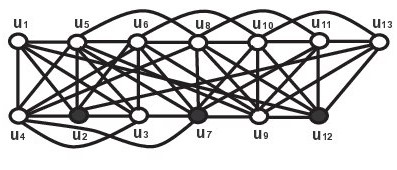}
  \caption{$T_{13}<1,2,3,4,11>$}\label{1}
\end{figure}
All three figures, explain the idea of IDN and black nodes have value of $1$.
  \begin{figure}[h]
  \centering
  \includegraphics[ width=.25\linewidth]{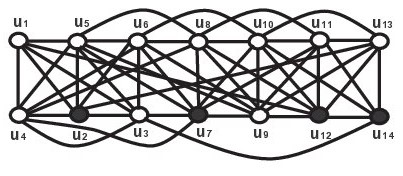}
  \caption{$T_{14}<1,2,3,4,11>$}\label{2}
\end{figure}

   \begin{figure}[h]
  \centering
  \includegraphics[ width=.25\linewidth]{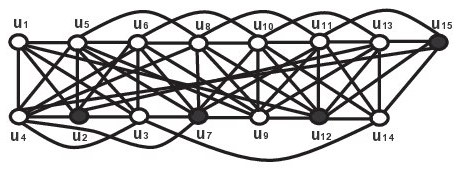}
  \caption{$T_{15}<1,2,3,4,11>$}\label{3}
\end{figure}

\end{proof}

\begin{theorem}
   Let for any $t_{j}$, $j = 1,2,3, \cdots, k$ and for all $ n $, when $t_{k} \equiv 0,2,3,4,5,\cdots, c_{i}-1 (mod $ $c_{i})$, $i\geq 4$ and $k = t_{m+3}, \cdots, n-1$ then $\Upsilon_I (T_{n}<t_{1},t_{2},t_{3}, \cdots, t_{m}, \cdots, t_{k}>) = \Big \lfloor \dfrac{n+c_{i}-2}{\mid c_{i}\mid} \Big \rfloor + 1 $
\end{theorem}

\begin{proof}
  Let $t_{j}$ be the generators from $t_{1},t_{2},t_{3}, \cdots, t_{m}, \cdots, t_{k}$ for $j = 1,2,3, \cdots, k$ where $k = t_{m}+3, t_{m}+4, \cdots, n-1$ and $c_{i}$ be the cardinality of total number of generators for $i\geq4$ in the graph. It means we need at least $4$ generators to start the process. Let start with initial value for which  $t_{j}$ varies from $6$ to $k<n$. Previous theorem gives explanation of Italian domination number for which values congruent to $1$ (mod $c_{i})$ and here these values are excluded. Since, $6 \equiv 2$ (mod $4)$, $7 \equiv 3$ (mod $4)$, $8 \equiv 0$ (mod $4)$, $10 \equiv 2$ (mod $4)$ and so on. Let,  $u_{1}, u_{2}, u_{3}, \cdots, u_{n}$ be the nodes and $c_{i}$ for $i = 4$ then for all possible values of $k$ goes from $6$ to $n-1$. Then Italian dominating function can be obtained by applying $f$ under value $1$ to f$(u_{1})$, f$(u_{5})$, f$(u_{9})$ $\cdots$, f$(u_{n})$. Hence, IDS = $\{u_{1}, u_{5}, u_{9}, \cdots, u_{n}\}$. \newline
  Moreover, apply function under same operation on any $c_{i}$, $i\geq 5$ which describe the cardinality of total number of generators. Take any $t_{j}$, $j = 1,2,3, \cdots, k$ and $k = t_{m}+3, t_{m}+4, \cdots, n-1$. Then the degree sequence of nodes for each graph is deg$(u_{1})$ = deg$(u_{n})$ = $c_{i}$, deg$(u_{2})$ = deg$(u_{n-1})$ = $c_{i}+1$, deg$(u_{3})$ = deg$(u_{n-2})$ = $c_{i}+2$, $\cdots$  deg$\dfrac{u_{n}+1}{2}$ for $n$ is odd and  deg$\dfrac{u_{n}}{2}$ for $n$ is even = $V_{\Delta_{G}}$ have same degree. Also, if the vertex $u_{4}$ is of degree $(c_{i}+2)$ then $u_{n-3}$ must have same degree. These properties make it unique and by using these properties Italian domination function can easily obtained. Thus, f$(u_{1})$ = f$(u_{t_{m}+2})$ = f$(u_{2t_{m}+3})$ = f$(u_{3t_{m}+4})$ = f$(u_{4t_{m}+5})$ = $\cdots$, f$(u_{{n}})$ = $1$ and all other remaining nodes in the graph labelled $0$. The function on these nodes helps to obtain desired IDN for any $c_{i}$, $i\geq 4$ and for all $t_{j}$. By using lemma $2$ and proposition $2$ all excluded values can be obtained. Consequently, IDN = $ \Big \lfloor \dfrac{n+c_{i}-2}{\mid c_{i}\mid} \Big \rfloor + 1 $ and IDS = $\{u_{1}, u_{t_{m}+2}, u_{2t_{m}+3}, u_{3t_{m}+4}, \cdots, u_{n}\}$.\newline
  \textbf{Conversely},

  Let $t_{1},t_{2},t_{3}, \cdots, t_{m}, \cdots, t_{k}$  be the generators and denoted as $t_{j}$ for  $j = 1,2,3, \cdots, k$. Let for the minimum generator values $c_{i} = 4$ to the maximum generator values $c_{i} = n-1$. Then for the different values of $t_{k}\not\equiv 1$ have values of $n$ that are $6 \equiv 2$ (mod $4)$, $7 \equiv 3$ (mod $4)$, $8 \equiv 0$ (mod $4)$ and so on. Since  deg$(u_{1})$ = deg$(u_{n})$ = $c_{i}$, deg$(u_{2})$ = deg$(u_{n-1})$ = $c_{i}+1$, deg$(u_{3})$ = deg$(u_{n-2})$ = $c_{i}+2$, $\cdots$  deg$\dfrac{u_{n}+1}{2}$ for $n$ is odd and  deg$\dfrac{u_{n}}{2}$ for $n$ is even = $V_{\Delta_{G}}$. Consider the function on f$(u_{1}) = 0$ and f$(u_{2}) = 1$ then sum of all neighbors of $u_{1}$ is not $2$. If the function on f$(u_{2}) = 0$ and f$(u_{3}) = 1$ then sum of all neighbors is not $2$. It means we need at least one more node to complete the process. So, the set generated through these nodes have $\lceil \dfrac{n}{4}\rceil$ $+ 1$ labelled nodes. But the minimum function can be obtained through $\lceil\dfrac{n}{4}\rceil$ nodes. Thus, any changing in the sequence cannot provide minimum Italian dominating set.

\end{proof}

\section{Conclusion}

This particular piece of writing sheds insight on the Italian dominance number of generalised Toeplitz graphs. Furthermore, we examine the minimum Italian domination sets for all generalized graphs constructed using different generators. These sets are sufficient to applying the domination function on the graph.


\begin{thebibliography}{22}

\bibitem{1}
	Abrishami, G., Henning, M.A.: Independent domination in subcubic graphs of girth at least six, Discrete Math. 341, no. 1, 155–164 (2018).
\bibitem{2}
	Behzad, A., Behzad, M., Cheryl, Praeger, C.E.: Fundamental dominations in graphs, arXiv, (2008).
\bibitem{3}
	Babikir, A., Henning, M.A.: Domination versus independent domination in graphs of small regularity, Discrete Math. 343, no. 7, 111727 (2020).
\bibitem{4}
Beeler, R.A., Haynes, T.W., Hedetniemi, S.T.: Double roman domination. Discret. Appl. Math. 211, 23–29 (2016).
\bibitem{5}
Bai, X., Gao, Z., Xi, C., Yue, J.:  On independent domination and packing numbers of subcubic graphs, arXiv, (2024).
\bibitem{6}
	Chellali, M., Haynes, T.W., Hedetniemi, S.T., McRae, A.A.: Roman {2}-domination. Discret. Appl. Math. 204, 22–28 (2016)
\bibitem{7}
Cho, E.K.,  Kim, J., Kim, M., Oum, S.I.: Independent domination of graphs with bounded maximum degree, J. Combin. Theory, Ser. B, Vol. 158, Part 2,  341-352 (2023)
	
\bibitem{8}
	Dorbec, P., Henning, M.A.,  Montassier, M., Southey, J.: Independent domination in cubic graphs, J. Graph Theory. 80 (2015), no. 4, 329–349 (2015)

\bibitem{9}
  Gao, H., Wang, P., Liu, E., Yang, Y.: More results on Italian domination in $C_{n}\ast C_{m}$, Mathematics, vol. 8, issue 4, 8040465 (2020)
\bibitem{10}
	Hao, G., Hu, K., Wei, S., Global, Z.X.: Italian domination in graphs. Quaest. Math. 41, 1–15 (2018).
\bibitem{11}
 Henning, M.A., Hedetniemi, S.T.: Defending the roman empire—a new strategy, Discrete Math. Vol. 266, 239-251 (2003)	
 \bibitem{12}
 Haynes, T.W., Henning, M.A.: Perfect Italian domination in trees. Discret. Appl. Math. 260, 164–177 (2019)
 \bibitem{13}
 Haynes, T.W., Hedetniemi, S.T., Henning, M.A.: Topics in domination in graphs, vol. 64. Springer Nature, Cham (2020)
  \bibitem{14}
 Leel, S., Srivastav, S.,  Ganesan, G.: Domination number in the context of some new graphs, Eng. Proceeding, 62, 14, 20240-62014 (2024)
\bibitem{15}
	Mojdeh, D.A., Masoumi, I., Volkmann, L.: Restrained double roman domination of a graph, Rairo Oper. Research, 2293–2304 (2022)
\bibitem{16}
Mojdeh, D.A., Volkmann, L.: Roman {3}-domination (double Italian domination). Discret. Appl. Math. 283, 555–564 ((2020))

\bibitem{17}
	Malik, S.: Hamiltonicity in directed Toeplitz graphs with $s_{1}=1$ and $s_{3}=4$, Comput. Mathem. Methods, 3676487 ( 2023)
\bibitem{18}
	Mojallal, S.A., Jung, J.H., Cheon, G.I., Kimb, S.R., Kang, B.: Structural properties of Toeplitz graphs, Discrete Mathematics, Vol. 345, Issue 11, 113016 (2022)

\bibitem{19}
	Nicoloso, S., Pietropaoli, U.: On the chromatic number of Toeplitz graphs, Discrete Applied Math. Vol. 164, Part 1, 286-296 (2014)
\bibitem{20}
	Nadeem, M.F., Qu, S., Ahmad, A., Azeem, M.: Metric dimension of some generalized families of Toeplitz graphs, Math. Problems in Engineering, 9155291 (2022)
\bibitem{21}
Prabhu, S., Arulmozhi, A.K., Arulperumjothi, M.: On power domination in certain chemical graphs. Int. J. Pure and Appl. Math. 118(11), 11–19 (2018)
\bibitem{22}
	Rahmouni, A., Chellali, M.: Independent roman {2}-domiantion in graphs. Discrete Appl. Math. 2018, 236, 408–414 (2018)
\bibitem{23}
	Rad, N.J., Volkmann, L.: A note on the independent domination number in graphs, Discrete Appl. Math. 161(18), 3087–3089 (2013)
	
\bibitem{24}
	Shabbir, A., Nadeem, M.F., Zamfirescu, T.: The property of hamiltonian connectedness in Toeplitz graphs, Complexity, 2020, 5608720 (2020)
\bibitem{25}
	van Dal, R., Tijssen, G., Tuza, Z., van der Veen, J., Zamfirescu, C.H., Zamfirescu, T.: Hamiltonian properties of Toeplitz graphs. Discret. Math. 159, 69–81 (1996)
\bibitem{26}
Yan, H., Kang, L., Xu, G.: The exact domination number of the generalized Petersen graphs, Discrete Math. 309, 2596–2607 (2009)
\bibitem{27}
	Zhao, M., Kang, L., Chang, G.J.: Power domination in graphs. Discrete Math. 306(15), 1812–1816 (2006)

\end{thebibliography}
\end{document}